\documentclass[oneside]{amsart}
\usepackage{amssymb, amscd}
\usepackage[all]{xy}
\usepackage{dsfont}

\usepackage{a4wide}
\usepackage[english]{babel}
\usepackage{amsfonts}
\usepackage{amsmath}
\usepackage{amsthm}
\title{Hilbert space compression for free products and HNN-extensions}
\author{Dennis Dreesen\\
Catholic University Leuven/ Universite de Neuchatel}
\thanks{The author is a research assistant for the Research Foundation - Flanders.}
\newtheorem{theorem}{Theorem}[section]
\newtheorem{definition}[theorem]{Definition}
\newtheorem{proposition}[theorem]{Proposition}
\newtheorem{lm}[theorem]{Lemma}
\newtheorem{corollary}[theorem]{Corollary}

\newtheorem{remark}[theorem]{Remark}
\newtheorem{notation}[theorem]{Notation}
\newcommand{\Z}{\mathbb{Z}}

\newcommand{\HNN}{\mbox{HNN}}
\newcommand{\N}{\mathbb{N}}
\newcommand{\R}{\mathbb{R}}
\newcommand{\h}{\mathcal{H}}

\newcommand{\x}{{\bf x}}
\newcommand{\Exp}{\mbox{Exp}}
\begin{document}
\maketitle
\begin{abstract}
Given the Hilbert space compression of two groups, we find bounds on the Hilbert space compression of their free product. We also investigate the Hilbert space compression of an HNN-extension of a group relative to a finite normal subgroup or a finite index subgroup. 
\end{abstract}
\section{Introduction}
Gromov introduced in \cite{Gromov} the notion of uniform embeddability of a metric space into a Hilbert space. Since uniform embeddability is a quasi-isometric invariant and since the word length metric on a finitely generated group is unique up to quasi-isometry, we obtain the notion of uniformly embeddable finitely generated group. Gromov believed that such groups would satisfy the Novikov Conjecture \cite{Gromov2}. Six years later, Yu came up with a formal proof of this claim \cite{Yu}. Moreover, together with Skandalis and Tu, he also proved that uniformly embeddable groups satisfy the coarse Baum-Connes Conjecture \cite{Yu:2}.

{\definition A metric space $(G,d)$ is {\em uniformly embeddable in a Hilbert space}, if there exist a Hilbert space $\h$, non-decreasing functions $\rho_-, \rho_+:\R^+ \rightarrow \R^+$ such that $\lim_{t\to \infty} \rho_-(t)=+\infty$, and a map $f:G\rightarrow \h$, such that
\[ \rho_-(d(x,y)) \leq d(f(x),f(y)) \leq \rho_+(d(x,y)) \ \forall x,y\in G.\]
The map $f$ is called a uniform embedding of $G$ in $\h$. It is called large-scale Lipschitz whenever $\rho_+$ can be taken of the form $\rho_+:t\mapsto Ct+D$ for some $C>0, D\geq 0$. It is Lipschitz if we can take $D=0$. \label{def:1}}\\

Assume for a moment that $(G,d)$ is a finitely generated group, equipped with the word length metric (always relative to some finite symmetric generating subset). This implies that $G$ is geodesic and uniformly discrete, and so given a uniform embedding $f:G\rightarrow \h$, one can always assume that it is Lipschitz (\cite{guekam}, Proposition 2.9). On the other hand, one can not always assume that $\rho_-$ is of the form $\rho_-:t\mapsto (1/C) t - D$ for some $C>0, D\geq 0$. Indeed, the free group on two generators is uniformly embeddable in a Hilbert space, but a theorem of Bourgain \cite{Bourgain} implies that it cannot be quasi-isometrically embedded in a Hilbert space. 

More generally,
fixing a uniform embedding $f:G\rightarrow \h$, Guentner and Kaminker \cite{guekam} introduced the notion of compression to measure {\em how close this embedding is to being quasi-isometric}. More precisely, the compression of a uniform embedding $f$, denoted $R(f)$, is defined as the supremum of all $\epsilon \in [0,1]$ for which there exist $C>0, D\geq0$ with  the property that
\[ d(f(x),f(y)) \geq (1/C) d(x,y)^{\epsilon} - D \ \forall x,y \in G.\]
Taking the supremum of $R(f)$ over all Hilbert spaces $\h$ and all uniform embeddings $f:G\rightarrow \h$, we obtain the Hilbert space compression, or shortly, the compression of $G$. We say that a uniform embedding $f:G\rightarrow \h$ is $G$-equivariant if there exists an affine isometric action $\mathcal{A}$ from $G$ on $\h$ such that $f(xy)=x\cdot_{\mathcal{A}} f(y) \ \forall x,y \in G$. We obtain the equivariant Hilbert space compression of $G$ as the supremum of $R(f)$, taken over all Hilbert spaces $\h$ and all equivariant uniform embeddings $f:G\rightarrow \h$. We remark that the (equivariant) compression is invariant under the choice of finite symmetric generating subset.

The above definitions for (equivariant) compression generalize naturally from the case of finitely generated groups equipped with the word length metric to the case of discrete groups whose metric is induced by some proper length function (see Definition \ref{def:lengthfunction}). In this case, it is no longer true that uniform embeddings are always Lipschitz and so we demand this explicitly in the definition for (equivariant) compression, i.e. we {\em only } take into account those (equivariant) uniform embeddings which are Lipschitz.

It is known that for every $\alpha \in [0,1]$, there exists an infinite, finitely generated group with compression $\alpha$ (see \cite{Arz}). Much less is known about the range of values of the equivariant compression. In fact, at the time of writing, the only known values for the equivariant compression are $0, 1/2 $ and $\frac{1}{2-2^{1-k}}$ with $k\geq 1$ (see also \cite{naoper}).

In \cite{guekam}, Guentner and Kaminker prove that the compression of a direct product of two groups $G_1$ and $G_2$ equals the minimum of the compressions of $G_1$ and $G_2$. Moreover, their proof also holds for the equivariant Hilbert space compression. In this article, we ask ourselves a similar question. Concretely, given the Hilbert space compressions of groups $G_1$ and $G_2$, we ask what can be said about the compression of the free product $G_1*G_2$ and what can be said about the  compression of an HNN-extension of $G_1$. We formulate a first result (see Theorem \ref{theorem:1} below). 

{\theorem Assume that $G_1$ and $G_2$ are finitely generated groups equipped with the word length metric. Denoting their Hilbert space compressions by $\alpha_1$ and $\alpha_2$ respectively, the Hilbert space compression $\alpha$ of the free product $G=G_1*G_2$, equipped with the word length metric, satisfies \[ \min(\alpha_1,\alpha_2,1/2)\leq \alpha \leq \min(\alpha_1,\alpha_2) .\]}

The result is generalized for free products of more general metric spaces, for free products, amalgamated over finite groups, and for HNN-extensions $\HNN(H,F,\theta)=<H,t\mid t^{-1} f t=\theta(f) \ \forall f\in F>$ where $F$ is a {\em finite} subgroup of $H$. We refer the reader to the end of Section \ref{sc:ordinarycompression} for details.\\

In Section \ref{sec:hnnfiniteindex}, we investigate the compression of an HNN-extension $\HNN(H,F,\theta)$ where both $F$ and $\theta(F)$ are finite index subgroups of $H$. We obtain the following bounds (See Theorem \ref{th:hnnfiniteindex}).

\begin{theorem}
Consider $G:=\HNN(H,F,\theta)$ where both $F$ and $\theta(F)$ are finite index subgroups of the finitely generated group $H$. Equip $G$ with the word length distance $d$ and $H$ with the induced metric $d_{in}$ from $G$. Then,
\[ \alpha_1/3 \leq \alpha \leq \alpha_1 , \] 
where $\alpha_1$ and $\alpha$ denote the Hilbert space compressions of $(H,d_{in})$ and $(G,d)$ respectively.
\end{theorem}

We will give a generalization to the case where $H$ is a group equipped with a proper length function (see Remark \ref{remark:HNNmetric}).  Regarding the equivariant Hilbert space compression, we obtain the following results in Section \ref{sc:equivariantcompression} (see Theorem \ref{theorem:equivembed} and Theorem \ref{theorem:hnnequiv} respectively).

{\theorem Let $G_1$ and $G_2$ be finitely generated groups equipped with the word length metric. Denote the equivariant Hilbert space compressions by $\alpha_1$ and $\alpha_2$ respectively. Denote $G=G_1*_F G_2$ an amalgamated free product where $F$ is a finite subgroup of both $G_1$ and $G_2$. If $\alpha$ denotes the equivariant Hilbert space compression of $G$, when equipped with the word length metric, then
\begin{enumerate}
\item $\alpha=1$ if $F$ is of index $2$ in both $G_1$ and $G_2$,
\item $\alpha=\alpha_1$ when $F=G_2$ and $\alpha=\alpha_2$ when $F=G_1$,
\item $ \alpha= \min(\alpha_1,\alpha_2, 1/2)$ otherwise.
\end{enumerate}}

Although condition $2.$ above is rather trivial, we have added it for completeness. For HNN-extensions over a finite group, we show
{\theorem Let $H$ be a finitely generated group with equivariant Hilbert space compression $\alpha_1$. Assume that $F$ is a finite subgroup of $H$ and that $\theta:F\rightarrow H$ is a group monomorphism such that the group generated by $\theta(F) \cup F$ is finite. Denote $G=HNN(H,F,\theta)$ and equip it with the word length metric. The equivariant Hilbert space compression $\alpha$ of $G$ satisfies
\begin{enumerate}
\item $\alpha=1$ whenever $F=H$,
\item $\alpha=\min(\alpha_1,1/2)$ otherwise.
\end{enumerate}}
Here, the first statement is rather trivial, but we have added it for completeness. Both of the above theorems generalize naturally to the case where $G_1, G_2$ (for the former result) and $H$ (for the latter result) are groups whose metrics are induced by proper length functions.
\section{Hilbert space compression of a free product of groups \label{sc:ordinarycompression}}
We start with a general lemma.
{\lm Let $\h$ be a Hilbert space and let $X$ be a uniformly discrete metric space, i.e. $B:= \inf\{d(x,y)\mid x,y\in X \}>0$. If $f:X\rightarrow \h$ is a map such that
\[ (1/C)\ d(x,y)^{\epsilon} -D \leq d(f(x),f(y)) \leq C\ d(x,y) +D \]
for some $\epsilon>0, C>0, D\geq 0$ and $\forall x,y\in X$, 
then there exist $\tilde{f}:X \rightarrow \h\oplus l^2(X)$ and a real number $\overline{C}>0$ such that
\begin{equation} \label{eq:unifembeddbis}
(1/\overline{C}) \  d(x,y)^{\epsilon} \leq d(f(x),f(y)) \leq \overline{C}\ d(x,y) \ \forall x,y\in X.
\end{equation}
\label{lm:Disnul}}
\begin{proof}
Define $\tilde{f}:X\rightarrow \h \oplus l^2(X), x\mapsto f(x)\oplus \delta_x$ where $\delta_x$ is the Dirac function at $x$. Then for every two distinct elements $x,y$ of $X$, we have that 
\[ \parallel \tilde{f}(x)-\tilde{f}(y) \parallel ^2= \parallel f(x)-f(y)\parallel ^2 + 2 .\]
Therefore
\[\parallel \tilde{f}(x)-\tilde{f}(y) \parallel \leq C d(x,y) + D + \sqrt{2} \leq (C+\frac{D+\sqrt{2}}{B}) d(x,y), \]
and so we obtain an upper bound like the one in equation (\ref{eq:unifembeddbis}) by setting
$\overline{C}=C+\frac{D+\sqrt{2}}{B}$.
With respect to the lower bound, we obtain that
\[ \parallel \tilde{f}(x)-\tilde{f}(y) \parallel \geq \frac{1}{\sqrt{2}} ( (1/C) d(x,y)^{\epsilon}-D ) +1 .\]
When $d(x,y)^{\epsilon}\geq 2CD$, then 
\[ \parallel \tilde{f}(x)-\tilde{f}(y) \parallel \ \geq \frac{1}{\sqrt{2}} (\frac{1}{2C} d(x,y)^{\epsilon}+\frac{1}{2C} d(x,y)^{\epsilon}-D) \geq (1/\overline{C}) d(x,y)^{\epsilon}, \]
if we take $\overline{C}\geq 2\sqrt{2}C$.
When $d(x,y)^{\epsilon}\leq 2CD$,
then 
\[\parallel \tilde{f}(x)-\tilde{f}(y) \parallel \geq 1 \geq \frac{1}{2CD} d(x,y)^{\epsilon} .\]

Finally, putting $\overline{C}:=\max(2\sqrt{2}C, 2CD,C+\frac{D+\sqrt{2}}{B})$, we obtain
\[ (1/\overline{C}) \ d(x,y)^{\epsilon} \leq \parallel \tilde{f}(x)-\tilde{f}(y) \parallel \leq \overline{C} \ d(x,y) \ \ \forall x,y\in X.\]

\end{proof}
\begin{definition}
A length function $l$ on a group $G$ is a function $l:G\rightarrow \R^+$ satisfying
\begin{enumerate}
\item $\forall x \in G: l(x)=0 \Leftrightarrow x=1$,
\item $\forall x\in G: \ l(x)=l(x^{-1})$ and
\item $\forall x,y\in G: \ l(xy)\leq l(x)+l(y)$.
\end{enumerate}
We say that $l$ is proper, whenever
  \[ \forall M \in \R^+: \ \mid \{ g\in G \mid l(g)\leq M \} \mid < \infty . \]
A length function on $G$ induces a metric on $G$ by setting $d(x,y)=l(x^{-1}y), \ \forall x,y\in G$. If $G$ is finitely generated and $S$ is a finite symmetric generating subset for $G$, then we can define the length of an element $x\in G$ as the length of the shortest path between $1$ and $x$ on the Cayley graph of $(G,S)$. Here, the length of every edge is counted as $1$. The so obtained length function is called the word length function and the associated metric is called the word length distance on $G$ relative to $S$. The (equivariant) Hilbert space compression is invariant under the choice of finite symmetric generating subset.
\label{def:lengthfunction}
\end{definition}

In this section, let us fix two finitely generated groups $G_1$ and $G_2$. We denote the chosen finite symmetric generating subset, the word length function, the word length distance and the Hilbert space compression of $G_1$ by $S_1, l_1, d_1$ and $\alpha_1$ respectively. We use similar notations $(S_2, l_2,d_2$ and $\alpha_2)$ for $G_2$.
We denote the generating subset $S_1\cup S_2$ on $G:=G_1*G_2$ by $S$ and equip $G$ with the word length metric relative to $S$.
Given $\alpha_1$ and $\alpha_2$, we would like to find bounds on the Hilbert space compression $\alpha$ of $G$. In order to do so, let us introduce some standard notations regarding free products of groups (see for example \cite{inbedmap}).

Two non-trivial elements $x,y$ of $G_1*G_2$ can always be written in reduced form as
\begin{equation}
\begin{split}
  x& =a_1 a_2 \ldots a_m \\
  y& =b_1 b_2 \ldots b_n,
 \end{split}
\end{equation}
where $m,n$ are natural numbers and where the $a_i,b_j$ are elements of $G_1\setminus \{1\} \bigcup G_2\setminus \{1\}$ such that no two consecutive elements $a_i, a_{i+1}$ or $b_j,b_{j+1}$ both belong to $G_1$ or both belong to $G_2$. If $i_0$ is the highest index such that $a_1, a_2, \ldots, a_{i_0-1}$ and $a_{i_0}$ are equal to $b_1, b_2, \ldots, b_{i_0-1}$ and $b_{i_0}$ respectively, then $h:= a_1 a_2 \ldots a_{i_0}$ is called the common part of $x$ and $y$. This way, we write
\begin{equation} \label{redschr:1}
\begin{split}
  x& =h\ g_x\ x_1\ x_2\ldots x_n \\
  y& =h\ g_y\ y_1\ y_2\ldots y_m,
 \end{split}
\end{equation}
where $h$ is the common part of $x$ and $y$. If $g_x\in G_1\setminus \{1\}$, and $g_y\in G_2 \setminus \{1\}$ (which can happen when $h=e$), then redefine $y_{m+1}:=y_m, y_m:=y_{m-1}, \ldots y_2:=y_1, y_1:=g_y, g_y:=1_{G_1}$, in order that $g_x,g_y$ both belong to $G_1$. Similar remarks hold when $g_x\in G_2\setminus \{1\}$ and $g_y\in G_1\setminus \{ 1 \}$. With the convention that empty sums are zero, we obtain
\begin{equation}\label{eq:distfrpr}
 d_G(x,y)=\sum_{i=1}^n l_{1,2}(x_i) + d_{1,2}(g_x,g_y)+ \sum_{j=1}^m l_{1,2}(y_j)
 \end{equation}
where $l_{1,2}$ stands for $l_{1}$ or $l_{2}$ as appropriate and similarly for $d_{1,2}$.\\

We prove the following theorem.
{\theorem \label{theorem:1} Assume that $G_1$ and $G_2$ are finitely generated groups equipped with the word length metric. Denoting their Hilbert space compressions by $\alpha_1$ and $\alpha_2$ respectively, the Hilbert space compression $\alpha$ of the free product $G=G_1*G_2$, equipped with the word length metric, satisfies \[ \min(\alpha_1,\alpha_2,1/2)\leq \alpha \leq \min(\alpha_1,\alpha_2) .\]}
{\remark The main part of the proof is based on Section $3$ of \cite{inbedmap}, where given uniform embeddings of the factor groups $G_1$ and $G_2$, an explicit uniform embedding of the free product $G_1*G_2$ is constructed.}
\begin{proof}
It is not hard to find the desired upper bound for $\alpha$ since $G_1$ and $G_2$ are metric subspaces of $G$ and so $\alpha \leq \min(\alpha_1,\alpha_2)$.

Next, choose a number $0\leq \epsilon < \min(\alpha_1,\alpha_2, 1/2)$ and let us take for each $i\in \{1,2\}$ a map $f_i : G_i \rightarrow \h:=l^2(\Z)$ which maps the identity element to $0$ and
such that
\begin{equation} \label{eq:unifembeddtris}
\exists C\geq 1, \ \forall x,y \in G_i :(1/C) d(x,y)^{\epsilon} \leq d(f_i(x),f_i(y)) \leq C d(x,y).
\end{equation}

Denote by $W_i$ ($i=1,2$) the set of those elements of $G$ whose expression as a reduced word begins with an element of $G_i$.
Notice that $W_1\cap W_2=\{1\}$. Define a Hilbert space $\overline{\h}$ by
\[ \overline{\h}=(\bigoplus_{W_1} \h) \oplus (\bigoplus_{W_2} \h). \]
Consider a map $f:G\rightarrow \overline{\h}$ defined as follows: set $f(e)=0$. Next, choose any element $x\in G\setminus \{1\}$ and write it as a reduced word $x=x_1 \ x_2 \ldots x_n$ where the $x_i$ are alternately elements of $G_1$ and $G_2$ (or $G_2$ and $G_1$). If $x_1 \in G_1$, then we define $f(x)=f(x)_1\oplus f(x)_2\in \overline{\h}$
by setting $f(x)_2=0$ and
\[ f(x)_1=\bigoplus_{h\in W_1} (f(x)_1)_h\mbox{ where } (f(x)_1)_h = \left\{ \begin{array}{ll}
         f_1(x_{2k+1}), & \exists k\geq 0 \mbox{ such that } h=x_1\ x_2 \ldots x_{2k}\\
         f_2(x_{2k}),  & \exists k\geq 1 \mbox{ such that } h=x_1 \ x_2 \ldots x_{2k-1} \\
		 0 & \mbox{otherwise} \end{array} \right. , \]
using the convention that an empty product is $1$. In particular, $(f(x)_1)_1=f_1(x_1), (f(x)_1)_{x_1}=f_2(x_2),
(f(x)_1)_{x_1x_2}=f_1(x_3),\ldots $. A similar formula is used when the reduced word expression of $x$ begins with an element of $G_2$. Let us show that
\[ (1/C) \ d(x,y)^{\min(\epsilon,1/2)} \leq d(f(x),f(y))\leq C \ d(x,y) \mbox{ for all $x,y\in G$.} \]

Choose $x,y\in G$ and write $x=hg_x x_1 x_2\ldots x_n$ and $y=h g_y y_1 y_2\ldots y_m$, where we use the same notations as in (\ref{redschr:1}). With the convention that empty sums are $0$, we obtain:
\[ \parallel f(x)-f(y) \parallel ^2= \sum_{i=1}^{n} \parallel f_{1,2}(x_i) \parallel ^2 + \parallel f_{1,2}(g_x)-f_{1,2}(g_y)\parallel ^2 + \sum_{j=1}^m \parallel f_{1,2}(y_j) \parallel ^2 .\]
By Equation (\ref{eq:unifembeddtris}), we obtain the upper bound
\[ \parallel f(x)-f(y) \parallel ^2 \leq \sum_{i=1}^{n} C^2 l_{1,2}(x_i)^2+ C^2 d_{1,2}(g_x,g_y)^2 + \sum_{j=1}^m C^2 l_{1,2}(y_j)^2 ,\]
and making use of Equation (\ref{eq:distfrpr}) and the fact that $\sqrt{a^2+b^2}\leq a+b$ for all $a,b\geq 0$, we obtain
\[ d(f(x),f(y))\leq \sqrt{ \sum_{i=1}^{n} C^2 l_{1,2}(x_i)^2 + C^2 d_{1,2}(g_x,g_y)^2 + \sum_{j=1}^m C^2 l_{1,2}(y_j)^2 }\leq C d(x,y) .\]
\noindent
On the other hand, we have that
\[ \parallel f(x)-f(y) \parallel \geq \sqrt{\sum_{i=1}^{n} (1/C^2) l_{1,2}(x_i)^{2\epsilon} + (1/C^2) d_{1,2}(g_x,g_y)^{2\epsilon} \ + \sum_{j=1}^m (1/C^2) l_{1,2}(y_j)^{2\epsilon}} \]
and since $a^{2\epsilon}+b^{2\epsilon}\geq (a+b)^{2\epsilon}$ for all $a,b\geq 0$, we get that
\[ d(f(x),f(y))\geq (1/C) \sqrt{d(x,y)^{2\epsilon}}=(1/C) d(x,y)^{\epsilon}, \]
which concludes the proof of the Theorem.
\end{proof}

Until now, we have only considered uniform embeddings of groups into Hilbert spaces, i.e. into $L^2-$spaces. By slightly modifying the definitions in the Introduction, one can choose any $p\geq 1$ and replace the class of $L^2-$spaces by the class of $L^p$-spaces. The proof of Theorem \ref{theorem:1} can easily be adapted to provide information about the $L^p$-compression of a free product of groups. Explicitly, we obtain
{\corollary Assume that $G_1$ and $G_2$ are finitely generated groups equipped with the word length metric. If we denote their $L^p$-compressions by $\alpha_1$ and $\alpha_2$ respectively, then the $L^p$-compression $\alpha$ of the free product $G=G_1*G_2$, when equipped with the word length metric, satisfies \begin{equation*} \min(\alpha_1,\alpha_2,1/p)\leq \alpha \leq \min(\alpha_1,\alpha_2). \end{equation*} 
\label{cor:Lpcompression}}
It is particularly interesting to notice that for $p=1$, we obtain that the $L^1$-compression of the free product $G_1*G_2$ equals the minimum of the $L^1$-compressions of $G_1$ and $G_2$.\\
Moreover, the same result holds when we replace the class of $L^1$-spaces, with a class $\mathcal{C}$ of Banach spaces which is stable under $l^1-$direct sum.

\begin{remark}
Given two metric spaces $(X_1,d_1), (X_2,d_2)$ and points $\tilde{x}_1\in X_1, \tilde{x}_2\in X_2$, let us define the free product of $(X_1,d_1,\tilde{x}_1)$ and $(X_2,d_2,\tilde{x_2})$. As a set, this space is equal to the collection of all words whose letters are alternately elements from $X_1\setminus \{\tilde{x}_1\}$ and $X_2\setminus \{\tilde{x}_2 \}$. We also include the word $\tilde{x}_1$ which we identify with $\tilde{x}_2$ (the idea being that $\tilde{x}_1$ and $\tilde{x}_2$ play the role of `` the identity elements'' of $X_1$ and $X_2$). Using $d_1$ and $d_2$, we define the distance between two elements
\[
  x =h g_x x_1 x_2 \ldots x_n \]
 \[  y =h g_y y_1 y_2\ldots y_m,  \]
written similarly as in Equation (\ref{redschr:1}), by
\[ \sum_{i=1}^n d_{1,2}(x_i,\tilde{x}_{1,2}) + d_{1,2}(g_x,g_y) + \sum_{j=1}^m d_{1,2}(y_j,\tilde{x}_{1,2}) ,\]
where $d_{1,2}$ stands for $d_1$ or $d_2$ and $\tilde{x}_{1,2}$ stands for $\tilde{x}_1$ or $\tilde{x}_2$ as appropriate.
We remark that Corollary \ref{cor:Lpcompression} generalizes to the context of free products of metric spaces under the assumption that the spaces $X_i$ are uniformly discrete, i.e. satisfy
\[ \inf\{d_{i}(x,y)\mid x,y\in X_{i} \}>0 , \ \forall i\in \{1,2\}. \]
\label{remark:freeproductmetricspaces}
\end{remark}


We end this section by a remark regarding the $L^p$-compression ($p\geq 1$) of $G_1*_F G_2$ and $\HNN(H,F)$ where $F$ is a finite subgroup and where $G_1$ and $G_2$ are finitely generated groups. Our claim is that the $L^p$-compressions of $G_1*_F G_2$ and $\HNN(H,F)$ are equal to that of $G_1*G_2$ and $G_1*\Z$ respectively. We distinguish two cases. First, assume that $G_1,G_2$ and $H$ are finite (or more generally are hyperbolic). It follows from \cite{hyper} that $G_1*_F G_2, G_1*G_2, H*\Z$ and $\HNN(H,F)$ are also hyperbolic and thus all have $L^p$-compression $1$ \cite{tess}. Secondly, assume that at least one of $G_1$ and $G_2$ is infinite. Then our claim regarding amalgamated products follows from Theorem $0.2$ in \cite{quasisom}, where it is shown that $G_1*_FG_2$ is quasi-isometric to $G_1*G_2$. Our claim regarding HNN-extensions follows from the same result, where it is proven that $\HNN(H,F)$ and $H*\Z$ are quasi-isometric for infinite $H$.


\section{Hilbert space compression of an HNN-extension over a finite index subgroup \label{sec:hnnfiniteindex}}
In \cite{guedad}, Guentner and Dadarlat give a definition of {\em uniform embeddability of a quasi-geodesic metric space into a Hilbert space} using the existence of certain families of unit vectors (see Proposition \ref{prop:quasigeo} below). 
\begin{definition}
A metric space $(X,d)$ is said to be quasi-geodesic if there exist $\delta>0$ and $\lambda\geq 1$ such that for every two points $x,y\in X$, there exists $n\in \N$ and a chain of elements $x_1,x_2, \ldots ,x_n$ such that
\[ d(x_i,x_{i+1})\leq \delta \ \forall i \in \{ 1, 2, \ldots , n-1 \} , \]
and
\[ \sum_{i=1}^n d(x_i,x_{i+1}) \leq \lambda d(x,y) .\]
\end{definition}
Note that any finitely generated group is quasi-geodesic when equipped with the word length metric relative to some finite symmetric generating subset.
{\proposition Let $(X,d)$ be a quasi-geodesic metric space. Then $X$ is uniformly embeddable in a Hilbert space if and only if for every $R>0$ and $\epsilon>0$ there exists a Hilbert space valued map $\xi:X\rightarrow \h, x\mapsto \xi_x$, such that $\parallel \xi_x \parallel=1$ for all $x\in X$ and such that
\begin{enumerate}
\item $\sup \{ \parallel \xi_x - \xi_{y} \parallel: d(x,y) \leq R, x,y\in X \} \leq \epsilon$,
\item $ \lim_{S\to \infty} \inf \{ \parallel \xi_x - \xi_{y} \parallel : d(x,y)\geq S, x,y\in X \} = \sqrt{2} $.
\end{enumerate} \label{prop:quasigeo}}
Note that the second condition is equivalent to
\[  \lim_{S\to \infty} \inf \{ \langle \xi_x,\xi_{y} \rangle : d(x,y)\geq S, x,y\in X \} = 0 .\]
\begin{remark}
The details of the proof of Proposition \ref{prop:quasigeo} will be important to us and so we include a slightly modified version of Guentner and Dadarlat's (\cite{guedad}, Proposition $2.1$) proof here. In fact, the above proposition holds for any, not necessarily quasi-geodesic, metric space. However, we give the proof for quasi-geodesic spaces because it will provide us with better lower bounds on the compression of $G$ later.
\end{remark}
\begin{proof} 
Assume that $X$ is uniformly embeddable and let $f:X\rightarrow \h$ be a uniform embedding of $X$ in a real Hilbert space $\h$. Let
\[ \Exp(\h)= \R \oplus \h \oplus (\h \otimes \h) \oplus (\h \otimes \h \otimes \h) \oplus \cdots \]
and define $\Exp:\h \rightarrow \Exp(\h)$ by
\[ \Exp(\zeta)=1 \oplus \zeta \oplus ( (1/\sqrt{2!}) \zeta \otimes \zeta ) \oplus ((1/\sqrt{3!}) \zeta \otimes \zeta \otimes \zeta ) \oplus \cdots .\]
Note that $\langle \Exp(\zeta),\Exp(\zeta') \rangle = e^{\langle \zeta,\zeta'\rangle}$, for all $\zeta,\zeta'\in \h$. For $t>0$ define 
\[ \xi_x=e^{-t \parallel f(x) \parallel^2} \Exp(\sqrt{2t} f(x)). \]
It is easily verified that $\langle \xi_x,\xi_y \rangle = e^{-t \parallel f(x)- f(y) \parallel^2}$. Consequently, for all $x,y\in X$ we have $\parallel \xi_x \parallel =1$, and
\begin{equation}
e^{-t \rho_+(d(x,y))^2} \leq \langle \xi_x, \xi_{y} \rangle \leq e^{-t \rho_-(d(x,y))^2} .
\end{equation}
Putting $t=\frac{-\ln(1-\frac{\epsilon^2}{2})}{\rho_+(R)^2}$,
it is easy to verify conditions $1.$ and $2.$ above.\\
Conversely, assume that $X$ satisfies the conditions in the statement and choose $p>0$. There exist a sequence of maps $\eta_n:X\rightarrow \h_n$ and a sequence of numbers $S_0=0<S_1 <S_2 < \ldots$ increasing to infinity, such that for every $n\geq 1$ and every $x,y\in X$,
\begin{enumerate}
\item $\parallel \eta_n(x) \parallel =1$
\item $\parallel \eta_n(x)- \eta_n(y) \parallel = \sqrt{2(1 -  \langle \eta_n(x), \eta_n(y) \rangle)}\leq \frac{1}{n^{1/2+p}}$, provided $d(x,y)\leq \ln(n)$,
\item $\langle \eta_n(x), \eta_n(y) \rangle \leq 1/2$ (or equivalently $\parallel \eta_n(x)-\eta_n(y) \parallel \geq 1$), provided $d(x,y)\geq S_n$.
\end{enumerate}
Choose a base point $x_0 \in X$ and define $f:X\rightarrow \bigoplus_{n=1}^{\infty} \h_n$ by
\[ f(x)=1/2 ( (\eta_1(x)-\eta_1(x_0)) \oplus ( \eta_2(x)-\eta_2(x_0) ) \oplus \cdots ). \]
It is not hard to verify that $f$ is well defined and
\[ \rho_-(d(x,y))\leq \parallel f(x)-f(y) \parallel \leq A d(x,y)+ B , \mbox{ for all } x,y\in X, \]
where $ A,B\in \R^+ $, $\rho_-=1/2 \sum_{n=1}^{\infty} \sqrt{n-1} \chi_{[S_{n-1},S_n)}$, and the $\chi_{[S_{n-1},S_n)}$ are the characteristic functions of the sets $[S_{n-1},S_n)$.

Indeed, let $x,y\in X$. If $n$ is such that $\ln(n-1)\leq d(x,y)<\ln(n)$, we have
\begin{eqnarray*}
\parallel f(x)- f(y) \parallel ^2 &=& \frac{1}{4} \sum_{i\leq n-1} \parallel \eta_i(x) -\eta_i(y) \parallel^2 + \frac{1}{4} \sum_{i\geq n} \parallel \eta_i(x) -\eta_i(y) \parallel^2 \\
& \leq & (n-1) + (1/4) \sum_{i\geq n} \frac{1}{i^{1+2p}} \leq e^{d(x,y)}+ C
\end{eqnarray*}
where $C=(1/4) \sum_{i\geq n} \frac{1}{i^{1+2p}} < \infty$. Since $X$ is a quasi-geodesic space, we conclude from Proposition $2.9$ in \cite{guekam} that 
\[ \parallel f(x)-f(y) \parallel \leq A d(x,y) + B, \]
for some $A,B\in \R^+$.

If $n$ is such that $S_{n-1}\leq d(x,y) <S_n$, we have
\[ \parallel f(x)- f(y) \parallel ^2 \geq \frac{1}{4} \sum_{i\leq n-1} \parallel \eta_i(x)-\eta_i(y) \parallel^2 \geq (n-1)/4=\rho_-(d(x,y))^2.\]
\end{proof}
\begin{remark}
Note that we did not use the fact that $(X,d)$ is quasi-geodesic in the first part of the proof.
\end{remark}
 
From now on, let $H$ be a finitely generated group and fix a finite symmetric generating subset $S$. Assume that $G:=\HNN(H,F,\theta)$ where both $F$ and $\theta(F)$ are of finite index in $H$. We equip 
$G$ with the word length metric relative to $S\cup \{t\}$ where $t$ is the stable letter of the HNN-extension. Adapting an idea of Guentner and Dadarlat (see the proof of Theorem $5.3$ in \cite{guedad}), we obtain a lower bound on the Hilbert space compression of $G$.
\begin{theorem}
Consider $G:=\HNN(H,F,\theta)$ where both $F$ and $\theta(F)$ are finite index subgroups of the finitely generated group $H$. Equip $G$ with the word length metric $d$ relative to some finite symmetric generating subset of $G$. Equip $H$ with the induced metric $d_{in}$ from $G$. Then,
\[ \alpha_1/3 \leq \alpha \leq \alpha_1 , \] 
where $\alpha_1$ and $\alpha$ denote the Hilbert space compressions of $(H,d_{in})$ and $(G,d)$ respectively.
\label{th:hnnfiniteindex}
\end{theorem}
To prove the above result, let us introduce some notations and definitions as in \cite{guedad}.
Recall that a {\em tree }consists of a set $V$ of vertices, a set $E$ of edges and two endpoint maps $E\rightarrow V$, associating to each edge its endpoints. Every two vertices in a tree are connected by a unique path without backtracking. Whenever two vertices $v,v'\in V$ are connected by an edge, then we denote this edge by $[v,v']$, or equivalently by $[v',v]$, i.e. edges do not carry an orientation.

A {\em tree of metric spaces} consists of families $(X_v)_{v\in V}$ and $(X_e)_{e\in E}$ of metric spaces and corresponding maps $\sigma_{e,v}:X_e \rightarrow X_v$ whenever $v$ is an endpoint of $e$. The maps $\sigma_{e,v}$ are called {\em structural maps}, the spaces $X_v$ are called {\em vertex spaces}  and the sets $X_e$ are called {\em edge spaces}.

Given an HNN-extension $G:=\HNN(H,F,\theta)$, we can use Bass-Serre theory to associate a tree $T$ to it as follows. As the set $V$ of vertices we take $G/H$, the collection of left cosets of $H$ in $G$. As the set $E$ of edges we take $G/F$, the left cosets of $F$ in $G$. Given $x\in G$, the edge $xF$ connects $xH$ and $xtH$. 

Notice that the vertices and edges of the above tree are actually subsets of $G$, so we can equip them as metric subspaces of $G$. Next, we can define structural maps $\sigma_{xF,xH}: xF \hookrightarrow xH$ by inclusion and $\sigma_{xF,xtH}:xF \rightarrow xtH$ by $xf\mapsto xft=xt\theta(f)$. This way, we obtain a tree of metric spaces which is called {\em the tree of metric spaces associated to the HNN-extension $G=\HNN(H,F,\theta)$}. Notice that the union of the vertex spaces equals $G$. 

\begin{remark}
There is a connection between the distance on $G$ and the distance on the underlying Bass-Serre tree. 
  Given $x,y\in G=\HNN(H,F,\theta)$, remark that the distance $d(x,y)$ in $G$ equals
  \[ d_T(xH,yH) + \inf \{ d(x_0,x_1) + d(x_2,x_3)+ \ldots + d(x_{p-1},x_p) \} ,\]
  where $d_T$ is the distance on the underlying tree $T$ (i.e. the number of edges in the shortest path connecting $xH$ and $yH$) and  where the infimum is taken over all sequences $x_0,x_1,\ldots , x_p$, where $p=2d_T(xH,yH)+1$
  and
  \begin{itemize}
  \item $x=x_0, y=x_p$
  \item $x_{2k}=x_{2k-1}t$ or $x_{2k}=x_{2k-1}t^{-1}$ for $k=1,\ldots , d_T(xH,yH),$
 \item $x_{2k}, x_{2k+1}$ lie in the same coset of $H$ for $k=0,1,\ldots,d_T(xH,yH)$.
 \end{itemize}
 \end{remark}
 Let us fix some notations. For a given vertex $v\in V$, we denote by $\alpha(v)\in V $ the unique vertex such that $[v,\alpha(v)]$ points towards the infinite geodesic $H,tH,t^2H, \ldots $. Here, just for this once, $[v,\alpha(v)]$ was considered as an {\em oriented} edge. Given vertices $v,v'\in V$, we denote by $(k,l)$ the unique pair of integers such that $\alpha^k(v)=\alpha^l(v')$ and $d_T(v,v')=k+l$.
Write $Y_v=\sigma_{[v,\alpha(v)],v}(X_{[v,\alpha(v)]})\subset X_v$ and remark that it is a left coset of $F$ or $\theta(F)$. Set $f_v=\sigma_{[v,\alpha(v)],\alpha(v)} \circ \sigma^{-1}_{[v,\alpha(v)],v}: Y_v \rightarrow X_{\alpha(v)}$. Finally,
let $Z>0$ be a real number such that every right coset of $F$ and $\theta(F)$ in $H$ contains a representative whose length is strictly smaller than $Z$.
 \begin{definition}
 Given $x_0\in G$, an {\em $s$-chain starting in $x_0$} is a sequence $\x=(x_0,x_1,\ldots ,x_{s-1})$ with $x_i\in X_{\alpha^i(v)}$ such that for each $0\leq i \leq s-2$ there exists $\overline{x_i}\in Y_{\alpha^i}(v)$ such that $d(x_i,\overline{x_i})=d(x_i,Y_{\alpha^i(v)})$ and $x_{i+1}=f_{\alpha^i(v)}(\overline{x_i})$.
 \end{definition}

\begin{lm}
Assume that $R$ is a strictly positive real number, let $x_0\in X_v$ and $x_0'\in X_{v'}$ with $d(x_0,x_0')<R$ and let $k,l$ and $Z$ be as just described. Then, any chains $(x_0,x_1, \ldots , x_k)$, $(x_0',x_1',\ldots , x_l')$ are such that
\[ \max\{(\sup_{0\leq i \leq k-1} d(x_i,x_{i+1}) ) , (\sup_{0\leq j \leq l-1} d(x_j',x_{j+1}') ), d(x_k,x_l') \} < (Z+2) R .\]
\label{lm:kenl}
\end{lm}

\begin{proof}

Fix $i\in \{0,1,\ldots , k-1\}$. Write $v_i=\alpha^i(v)$, denote $e=[v_i,\alpha(v_i)]$ and take $a\in G$ such that $X_e=aF$. This implies either that $X_{v_i}=aH$ and $X_{\alpha(v_i)}= atH$ or that $X_{v_i}=atH$ and $X_{\alpha(v_i)}=aH$. We only prove the second case, leaving the first case as an exercise to the reader.

The elements of $Y_{v_i}$ are of the form $aft=at\theta(f)$ where $f\in F$. Writing $x_i=ath$ for some $h\in H$, take $b$ a representative of $\theta(F)h$ whose length is smaller than $Z$. Then clearly $d(x_i,x_{i+1}) =d(x_i,Y_{v_i}) + 1 \leq d(x_i,x_ib^{-1})+1 < Z +1 <(Z+2)R$. 

Analogously, one proves that $d(x_j',x_{j+1}') <  Z+1 < (Z+2)R$ for $j\in \{0,1,\ldots ,l-1\}$.
For the case $i=k$, we use the triangle inequality to get that
\[ d(x_k,x_l')<(Z+1) (k+l) + R \leq (Z+2)R .\]
\end{proof}

\begin{notation}
Given $R>0$ and $\epsilon > 0$, choose and fix $s, n\in \N_0$ such that
\begin{equation}
\sqrt{2/s}\leq \frac{\epsilon}{2(R+1)},  \ n\geq (Z+2)R. 
\label{eq:voorwaarde1}
\end{equation}

Next, using Proposition \ref{prop:quasigeo}, find a Hilbert space $\overline{\h}$ and
unit vectors $\{ \tilde{\xi}_x \mid x\in H \} \subset \overline{\h}$ satisfying the conditions
\begin{equation}
\sup\{ \parallel \tilde{\xi}_y -\tilde{\xi}_y' \parallel : d(y,y')\leq n+2s(Z+1) \} \leq  \frac{\epsilon}{2(R+1)} ,
\end{equation}
\begin{equation}
\lim_{S\to \infty} \sup\{ \mid \langle \tilde{\xi}_y, \tilde{\xi}_y' \rangle \mid : d(y,y')\geq S \} = 0.
\end{equation}
For each $v\in V$, denote $\h_v:=\overline{\h}$. Since $G$ is the disjoint union of the vertex spaces $X_v$, we can take unit vectors $\{\xi_x \mid x\in G \} \subset \h:= \oplus_{v\in V} \h_v$ such that $\xi_x\in \h_v$ whenever $x\in X_v$ and such that
\begin{equation}
\sup\{\parallel \xi_y -\xi_y' \parallel : d(y,y')\leq n+2s(Z+1), y,y'\in X_v, v\in V \} \leq  \frac{\epsilon}{2(R+1)} ,
\label{eq:convergence}
\end{equation}
\begin{equation}
\lim_{S\to \infty} \sup\{ \mid \langle \xi_y, \xi_y' \rangle \mid: d(y,y')\geq S, y,y'\in X_v, v\in V \} = 0.
\label{eq:support}
\end{equation}

Finally, for every $s$-chain $\x=(x_0,x_1, \ldots ,x_{s-1})$, define the unit vector $\eta^{\x}\in \h$ by
\begin{equation}
\eta^{\x} = \sqrt{1/s} \sum_{i=0}^{s-1} \xi_{x_i} .
\label{eq:defvector}
\end{equation}
\label{def:maindefinition}
\end{notation}
Our initial goal is to prove Proposition \ref{prop:hoofdzaak}, namely that the vectors $\eta^{\x}$ satisfy properties similar to those of Proposition \ref{prop:quasigeo}.
\begin{lm}
Let $\x=(x_0,x_1,\ldots ,x_{s-1})$ and $\x'=(x_0',x_1',\ldots ,x_{s-1}')$ be $s$-chains starting in $X_v$. If $d(x_0,x_0')\leq n$, then $\parallel \eta^{\x} - \eta^{\x'} \parallel \leq \frac{\epsilon}{2(R+1)}$.
\label{lm:converging}
\end{lm}
\begin{proof}
We write
\begin{equation}
\parallel \eta^{\x} - \eta^{\x'} \parallel = \parallel \sqrt{1/s} \sum_{i=0}^{s-1} (\xi_{x_i} - \xi_{x_i}') \parallel    . \label{eq:sameclass}   
\end{equation}
Since by the triangle inequality $d(x_i,x_i')\leq n+ 2i(Z+1) \leq  n+ 2s(Z+1)$,
we can bound (\ref{eq:sameclass}) by
\[ \sup_{0\leq i \leq s-1} \parallel \xi_{x_i}- \xi_{x_i'} \parallel \leq \sup_{\Delta} \parallel \xi_y- \xi_{y'} \parallel \leq \frac{\epsilon}{2(R+1)}, \]
where $\Delta=\{ (y,y') \mid d(y,y') \leq n+2s(Z+1), y,y'\in X_v, v\in V \}$.
 \end{proof}
 \begin{lm}
Let $\x=(x_1,\ldots ,x_s)$ and $\x'=(x_0,x_1',x_2',\ldots ,x_{s-1}')$ be $s$-chains with $x_0\in X_v$ and $x_1\in X_{\alpha(v)}$. If $(x_0,x_1)$ is a $2$-chain and $d(x_0,x_1)\leq n$, then $\parallel \eta^{\x}-\eta^{\x'} \parallel \leq \frac{\epsilon}{R+1}$.

\label{lm:converging2}
\end{lm}

\begin{proof}
Denote $\overline{x}_0=f_v^{-1}(x_1)$ and set $\overline{\x}=(\overline{x}_0,x_1, \ldots ,x_{s-1})$. Then
\[ \eta^{\x}=\eta^{(x_1,x_2,\ldots ,x_s)}= \sqrt{1/s} (\sum_{i=1}^{s-1} \xi_{x_i} +  \xi_{x_s} ), \]
\[ \eta^{\overline{\x}}=\eta^{(\overline{x}_0,x_1,\ldots ,x_{s-1})} = \sqrt{1/s} (\xi_{\overline{x}_0}+ \sum_{i=1}^{s-1} \xi_{x_i} ) . \]

Therefore, 
\[ 
\parallel \eta^{\overline{\x}} - \eta^{\x} \parallel  =  \sqrt{1/s} \parallel  \xi_{\overline{x}_0} - \xi_{x_s} \parallel = \sqrt{2/s} \leq \frac{\epsilon}{2(R+1)}, \]
where the final inequality comes from the choice of $s$ in Equation (\ref{eq:voorwaarde1}).
Since $d(x_0,\overline{x}_0) = d(x_0,x_1)-1 \leq n$, we can apply Lemma \ref{lm:converging} to the chains $\x'$ and $\overline{\x}$ to conclude that
\[ \parallel \eta^{\x} - \eta^{\x'} \parallel \leq \parallel \eta^{\x} - \eta^{\overline{\x}} \parallel + \parallel \eta^{\overline{\x }} - \eta^{\x'} \parallel \leq  \frac{\epsilon}{2(R+1)} + \frac{\epsilon}{2(R+1)}= \frac{\epsilon}{R+1 } . \]

\end{proof}

\begin{lm}
For any $2$ $s$-chains $\x=(x_0,x_1, \ldots ,x_{s-1})$ and $\x'=(x_0',x_1',\ldots ,x_{s-1}')$,
\[ \mid \langle \eta^{\x}, \eta^{\x'} \rangle \mid \leq \sup\{\langle \xi_y, \xi_{y'} \rangle \mid : y,y'\in X_v, v\in T, d(y,y')\geq d(x_0,x_0')-2s(Z+1) \} .\]
\label{lm:convergentie}
\end{lm}

\begin{proof}
Assume that $x_0\in X_{\tilde{v}}$ and $x_0'\in X_{v'}$. As before, denote $(k,l)$ the unique pair of natural numbers such that $d_T(\tilde{v},v')=k+l$ and $\alpha^k(\tilde{v})=\alpha^l(v')$. By symmetry, we will assume that $k\geq l$. Further, we will assume that $k<s, $ because $k\geq s$ implies that $\langle \eta^{\x} ,\eta^{\x'} \rangle = 0 $. We obtain by definition that

\[ \langle \eta^{\x}, \eta^{\x'} \rangle = (1/s) \sum_{i=0}^{s-k-1}  \langle \xi_{x_{k+i}} , \xi_{x_{l+i}'} \rangle .\]
Notice that $d(x_{k+i},x_{l+i}') \geq d(x_0,x_0') - (k+i) (Z+1) - (l+i) (Z+1) \geq d(x_0,x_0') - 2s(Z+1)$, so
\[ \mid \langle \eta^{\x}, \eta^{\x'} \rangle \mid \leq    \sup_\Omega \mid \langle \xi_y, \xi_y' \rangle \mid   ,     
\]
where $\Omega=\{ (y,y') : y,y'\in X_v, v\in T, d(y,y')\geq d(x_0,x_0')-2s(Z+1) \}$.

\end{proof}

\begin{proposition}
Given $R>0$ and $\epsilon>0$, let $s $ and $(\xi_x)_{x\in G}$ be constructed as in Definition \ref{def:maindefinition}. For each $x_0\in G$, choose and fix an $s$-chain $\x=(x_0,x_1, \ldots ,x_{s-1})$ and consider the corresponding vector $\eta^{\x}=\eta^{(x_0,x_1,\ldots x_{s-1})}$. Then
\begin{equation}
\sup\{\parallel \eta^{\x} - \eta^{\x'} \parallel : d(x_0,x_0') < R \} \leq \epsilon,
\label{eq:result1}
\end{equation}
and

\begin{equation}
\mid \langle \eta^{\x}, \eta^{\x'} \rangle \mid \leq \sup\{ \mid \langle \xi_y, \xi_{y'} \rangle \mid : y,y'\in X_v, v\in V, d(y,y')\geq d(x_0,x_0') - 2s(Z+1) \} .
\label{eq:result2}
\end{equation}
\label{prop:hoofdzaak}
\end{proposition}
\begin{proof}

Condition (\ref{eq:result2}) was proven in Lemma \ref{lm:convergentie}. To prove (\ref{eq:result1}), let us choose $x_0,x_0'\in G$ such that $d(x_0,x_0') < R$. Choose any two $s$-chains $\x$ and $\x'$ starting at $x_0$ and $x_0',$ respectively. We want to prove that $\parallel \eta^{\x} - \eta^{\x'} \parallel \leq \epsilon$. 

Therefore, let $k$ and $l$ be as before. Take chains $(x_0,x_1, \ldots ,x_k)$ and $(x_0',x_1', \ldots ,x_l')$. By Lemma \ref{lm:kenl}, we have that
\[ \max\{(\sup_{0\leq i \leq k-1} d(x_i,x_{i+1}) ) , (\sup_{0\leq j \leq l-1} d(x_j',x_{j+1}') ), d(x_k,x_l') \} < (Z+2) R \leq n ,\]
where the last inequality follows from Equation (\ref{eq:voorwaarde1}). Consequently, we can apply Lemma \ref{lm:converging} and Lemma \ref{lm:converging2} repeatedly to chains  $\x(i)$ and $\x'(j)$ whose initial elements are $x_i, i=0\ldots k$ and $x_j', j=0\ldots l$ ,  respectively. We obtain
\begin{eqnarray*}
\parallel \eta^{\x} - \eta^{\x'} \parallel & \leq & \sum_{i=0}^{k-1} \parallel \eta^{\x(i)} - \eta^{\x(i+1)} \parallel + \parallel \eta^{\x(k)} - \eta^{\x'(l)} \parallel + \sum_{j=0}^{l-1} \parallel \eta^{\x'(j)} - \eta^{\x'(j+1)} \parallel \\
& \leq & \frac{(k+l+1) \epsilon}{R+1} \leq \epsilon .
\end{eqnarray*}
\end{proof}
\noindent We are now ready to conclude the proof of Theorem \ref{th:hnnfiniteindex}.\\
{\bf Proof of Theorem \ref{th:hnnfiniteindex}.}\\
Clearly $\alpha\leq \alpha_1$ since $(H,d_{in})$ embeds isometrically in $(G,d)$.

Conversely, assume that $\alpha_1>0$ and fix any real number $0<p<\alpha_1$. Next, choose $C>0$ and $D\geq 0$ such that there exists a uniform embedding $f$ of $(H,d_{in})$ in a Hilbert space satisfying 
\[ \rho_-(d_{in}(x,y)):= (1/C) d_{in}(x,y)^{\alpha_1-p} - D\leq d(f(x),f(y)) \leq C d_{in}(x,y)+D := \rho_+(d_{in}(x,y)) \ \forall x,y \in H. \]
For each $m\in \N_0$, define $\epsilon_m= \frac{1}{m^{1/2+p}}, R_m=\ln(m)$ and define $n_m=m^p$ and $s_m=m^{1+6p}$. Clearly then
\begin{equation}
 n_m \geq  (Z+2)R_m , \ \sqrt{2/s_m} \leq \frac{\epsilon_m}{2(R_m+1)} ,
\label{eq:voorwaarden}
\end{equation}
whenever $m$ is larger than some natural number $r_p$. Next, find a collection of unit vectors $\{ \xi_y \mid y\in H \}$ in some Hilbert space $\overline{\h}$ such that
\[ \parallel \xi_y - \xi_{y'} \parallel  \leq \frac{\epsilon_m}{2(R_m+1)} = \frac{1}{2m^{1/2+p}(\ln(m)+1)} :=\overline{\epsilon}_m \mbox{ when } d(y,y')\leq n_m+2s_m(Z+1) :=\overline{R}_m ,\]
by setting $t_m= \frac{-\ln(1-\frac{1}{2}\overline{\epsilon}_m^2)}{(C \overline{R}_m + D)^2}$ in the proof of Proposition \ref{prop:quasigeo}.
We should denote these vectors by $\xi^m_y$, but we drop the upper index to lighten notation.
Notice that the so obtained vectors also satisfy 
\[ \mid \langle \xi_y, \xi_{y'} \rangle \mid  \leq (1-\frac{1}{2}\overline{\epsilon}_m^2)^{\frac{\rho_-(d(x,y))^2}{(C\overline{R}_m+D)^2}} = (1-\frac{1}{2}\overline{\epsilon}_m^2)^{\frac{((1/C) d(y,y')^{\alpha_1-p} - D)^2}{(C\overline{R}_m+D)^2}} .\]
Therefore, $ \mid \langle \xi_y, \xi_{y'} \rangle \mid \leq 1/2$ whenever
\[ (1-\frac{1}{8m^{1+2p} (\ln(m)+1)^2})^{\frac{((1/C) d(y,y')^{\alpha_1-p} - D)^2}{(C\overline{R}_m+D)^2}} \leq 1/2 ,\]
 and so whenever
\[  {\frac{((1/C) d(y,y')^{\alpha_1-p} - D)^2}{(C\overline{R}_m+D)^2}}  \geq \frac{-\ln(2)}{\ln(1-\frac{1}{8m^{1+2p} (\ln(m)+1)^2})} .\]
Using the fact that $\lim_{m\to \infty} \frac{-1}{\ln(1-\frac{1}{8m^{1+2p} (\ln(m)+1)^2})}  \frac{1}{8m^{1+3p}}=0$, we see that for $m$ larger than some natural number $\tilde{r}_p$, the above inequality is true if  
\[ d(y,y')^{\alpha_1-p} \geq \sqrt{ln(2)8m^{1+3p}} (C\overline{R}_m +D) C + CD .\] We conclude that for $m$ larger than $\delta(p):=\max(r_p,\tilde{r}_p)$, Equation (\ref{eq:voorwaarden}) holds and

\[ \mid \langle \xi_y, \xi_{y'} \rangle \mid \leq 1/2  \mbox{ whenever } d(y,y')\geq m^{\frac{\frac{3}{2}+9p}{\alpha_1-p}} .\]

Denote $S_m=m^{\frac{\frac{3}{2}+9p}{\alpha_1-p}}$. For every $x_0, x_0'\in G$ and for every $m\in \N_0$ larger than $\delta(p)$, Proposition \ref{prop:hoofdzaak} gives vectors $(\eta^{\x}_m),\ (\eta^{\x'}_m) \in \h:=\oplus_{v\in V} \overline{\h}$ where $\x$ and $\x'$ are $s_m$-chains starting in $x_0$ and $x_0'$ respectively. Moreover,
\[
\sup\{\parallel \eta^{\x}_m - \eta^{\x'}_m \parallel : d(x_0,x_0') < R_m \} \leq \epsilon_m \]
and
\[
\mid \langle \eta^{\x}_m, \eta^{\x'}_m \rangle \mid \leq 1/2 \mbox{ whenever } d(x_0,x_0')\geq S_m + 2s_m(Z+1). \]

Denote $S_m'=S_m+2s_m(Z+1)$. Following the proof of Proposition \ref{prop:quasigeo}, we get the existence of a uniform embedding of $G$ into a Hilbert space, whose compression map $\rho_-$ is greater or equal to
\[ 1/2 \sum_{m=1}^{\infty} \sqrt{m-1-\delta(p)} \  \chi_{[S_{m-1}',S_m'[} \mbox{ on } [S_{\delta(p)}',\infty[. \]
The compression of $G$ is greater than $\beta$, whenever $(S_m')^{\beta}$ eventually (i.e. for $m$ large enough) lies under $\sqrt{m-1-\delta(p)}$. This is certainly true if $m^{\beta \frac{\frac{3}{2}+10p}{\alpha_1-p}}\leq m^{1/2}$ for $m$ large enough, which amounts to saying that
\[ \beta \leq \frac{\alpha_1-p}{3+18 p} .\]

Recalling that $p>0$ can be taken arbitrarly small, we can let $p$ go to $0$ and obtain that the compression of $G$ is at least $\alpha_1/3$. \begin{flushright}$\Box$ \end{flushright}
\begin{remark}
Assume that $H$ is a group equipped with a proper length function $l_1$.  Once again, we let $G:=\HNN(H,F,\theta)$ be an HNN-extension of $H$ such that both $F$ and $\theta(F)$ are of finite index in $H$. There is a natural way to equip $G$ with a length function. Indeed, each element $g\in \HNN(H,F,\theta)$ can be represented by a collection of words $(a_1,t^{i_1},a_2,t^{i_2},\ldots , a_{n-1}, t^{i_{n-1}}, a_n)$ where $n$ runs over the natural numbers, where $g=a_1t^{i_1}a_2t^{i_2}\cdot \ldots \cdot a_n$, where the $a_i$ belong to $H$ and where $i_1,i_2,\ldots i_{n-1}\in \{1,-1\}$. We define the length of such a word as the sum $l_1(a_n)+ \sum_{j=1}^{n-1} (l_1(a_i) + \mid i_j \mid)$. We define the length of $g$ as the infimum of the lengths of all words representing $g$. It is easy to see that this defines a proper length function on $G$.\\
The crucial difference with the previous case of finitely generated groups equipped with the word length metric, is the fact that $G$ need no longer be quasi-geodesic. We are therefore obliged to replace $R_m=\ln(m)$ by $\sqrt{m}$ in the proof of Proposition \ref{prop:quasigeo}. Without many changes to the proofs of the lemmas, but with straightforward adaptations made to the proof of Theorem \ref{th:hnnfiniteindex}, we obtain the following result:\\
{\em
Let $H$ and $G$ be as just described. Then,
\[ \alpha_1/6 \leq \alpha \leq \alpha_1 , \] 
where $\alpha$ denotes the Hilbert space compressions of $G$ and where $\alpha_1$ denotes the Hilbert space compression of $H$, when equipped with the induced metric from $G$.}
\label{remark:HNNmetric}
\end{remark}

\section{Equivariant Hilbert space compression for a free product of finitely generated groups \label{sc:equivariantcompression}}
Let us begin this section by recalling some elementary facts about affine isometric actions. Any affine isometric action $\chi$ of a group $H$ on a Hilbert space $\h$ can be written as 
\[ \chi(x) v = \pi(x) v + b(x) \ \forall v\in \h , \ \forall x\in H, \]
where $\pi: H \rightarrow \mathcal{O}(\h)$ is a group homomorphism from $H$ to the orthogonal group of $\h$ and where $b:H\rightarrow \h$ is a map satisfying the $1$-cocycle inequality, i.e.
\[ b(xy)=\pi(x)(b(y))+ b(x) \ \forall x,y\in H. \]
We summarize some standard properties of $1-$cocyles in the lemma below.
{\lm Let $H$ be a group. Assume as above that $\chi$ is an affine isometric action with associated $1$-cocycle $b:H \rightarrow \h$. Then,
\begin{enumerate}
\item $\parallel b(x)-b(y) \parallel=\parallel b(y^{-1}x) \parallel \ \forall x,y\in H$,
\item if $f:H\rightarrow \h$ is a $H-$equivariant map relative to $\chi$ and the left multiplication action by $H$ on itself, then the compression of $b$ equals the compression of $f$.
\end{enumerate} \label{lm:factscocycles}}
\begin{proof}
We leave the proof as an exercise to the reader.
\end{proof}

\begin{definition}
A continuous  function $\psi:H\rightarrow \R^+$ is {\em conditionally negative definite} if $\psi(x)=\psi(x^{-1})$ for all $x\in H$ , $\psi(1)=0$, and
\[ \sum_{i=1}^{n} \sum_{j=1}^{n} a_i a_j \psi(x_i^{-1}x_j) \leq 0, \]
for all $n\in \N_0$, $x_1,x_2,\ldots ,x_n \in H$ and all $a_1,a_2,\ldots ,a_n \in \R$ with $\sum_{i=1}^n a_i=0$.
\end{definition}
 
The proofs of the following results can be found in \cite{propriété(T)} (p.$63$) and \cite{Haagerup} (Lemma $6.2.1$) respectively.
{\theorem Assume that $\psi:H\rightarrow \R^+$ is a conditionally negative definite function on a topological group $H$. There exists a real Hilbert space $\h_{\psi}$, an orthogonal representation $\pi_{\psi}:H\ \rightarrow \mathcal{O}(\h)$ and a $1-$cocycle $b_{\psi}$ associated to $\pi_{\psi}$ such that
\[ \psi(x)=\parallel b_{\psi}(x) \parallel ^2, \]
for every $x\in H$.\\
Conversely, if $b$ is a $1-$cocycle relative to some affine isometric action of $H$ on a Hilbert space, then the map $\psi: x\mapsto \parallel b(x) \parallel ^2$ is conditionally negative definite. \label{theorem:cocconnegdef}}
{\lm Let $F$ be an open compact subgroup of the locally compact group $H$. If $\psi$ is a continuous, conditionally negative definite function on $H$, then there exists a continuous conditionally negative definite function $\psi'$ on $H$ such that
\begin{enumerate}
\item $\psi'$ is F-bi-invariant, (that is,
\[ \psi'(fxf')=\psi'(x) \ \forall x\in H \ \forall f,f'\in F) ;\]
\item $\psi'(f)=0 \ \forall f\in F$, and $\psi'(x)\geq 1$ for all $x \in H\setminus F;$
\item $\psi - \psi'$ is bounded.
\end{enumerate} \label{lm:positivecocycle}}
Notice that from $(2)$ in Lemma \ref{lm:factscocycles}, Theorem \ref{theorem:cocconnegdef} and $(2),(3)$ of Lemma \ref{lm:positivecocycle} (take $F=\{1\}$), it follows that the equivariant compression of a locally compact group $H$ is the supremum of all $\epsilon \in [0,1]$ such that there exist $C>0$ and a $1$-cocycle $b$,
satisfying
\begin{equation}
(1/C) \ l(x)^{\epsilon} \leq \parallel b(x) \parallel \leq C \ l(x), \ \forall x \in H .
\label{eq:compco}
\end{equation}
The following is standard.
{\lm Denote $H$ any topological group equipped with some length function $l$ and let $F$ be a finite normal subgroup of $H$. If we define the length of an element $\overline{x}$ of $H/F$ as the minimum of $l(y)$ where $y\in xH$, then the equivariant Hilbert space compressions of $H$ and $H/F$ are equal. \label{cor:quotient}}
\begin{proof}
Given a $1$-cocycle $b: H \rightarrow \h$ which is Lipschitz, we get a conditionally negative definite map $\psi: x\mapsto \parallel b(x) \parallel ^2 $. By Lemma \ref{lm:positivecocycle}, there exists a conditionally negative definite function, $\psi'$, at bounded distance from $\psi$ such that 
\begin{enumerate}
\item $\psi'$ is $F$-bi-invariant;
\item $\psi'(f)=0 \ \forall f\in F$, and $\psi'(x)\geq 1$ for all $x \in H\setminus F$.
\end{enumerate}
This implies that the map $\psi'$ is in fact a conditionally negative definite function on $H/F$. The $1-$cocycle $b'$ associated to $\psi'$ by Theorem \ref{theorem:cocconnegdef} has the same compression as $b$ and is again a Lipschitz map.

Conversely, starting with a Lipschitz $1-$cocycle $b':H/F \rightarrow \h$, we obtain a conditionally negative definite function $\psi'(\overline{x})=\parallel b'(\overline{x}) \parallel ^2$. Define $\psi:H\rightarrow \R^+$ by setting $\psi(x)=\psi'(\overline{x})$. This map is clearly conditionally negative definite and so Theorem \ref{theorem:cocconnegdef} associates a $1-$cocycle $b$ to it. It is clear that the compressions of $b$ and $b'$ are again equal and that $b$ is Lipschitz.
\end{proof}

\begin{remark}
Assume that $G_1$ and $G_2$ are countable groups equipped with proper length functions $l_1$ and $l_2$ respectively. There is a natural way to equip an amalgamated free product $G:=G_1 *_C G_2$ with a proper length function. Indeed, each element $g\in G$ can be represented by a collection of words $(a_1,a_2,a_3,\ldots , a_n)$ where $n$ runs over the natural numbers, where $g=\prod_{i=1}^{n} a_i$ and where the $a_i$ belong alternately to $G_1$ and $G_2$. We define the length of a word as the sum $l_1(a_1)+l_2(a_2)+l_1(1_3) + \ldots +l_{1,2}(a_n)$ where $l_{1,2}$ is $l_1$ or $l_2$ as appropriate. We define the length of $g$ as the minimum of the lengths of all words representing $g$. It is easy to check that this defines a proper length function on $G$.
\label{remark:lengthamalga}
\end{remark}

We are ready to prove the following theorem.
{\theorem Let $G_1$ and $G_2$ be countable groups equipped with proper length functions $l_1$ and $l_2$ respectively. Denote the equivariant Hilbert space compressions of $G_1$ and $G_2$ by $\alpha_1$ and $\alpha_2$. Denote $G=G_1*_F G_2$ an amalgamated free product where $F$ is a finite subgroup of both $G_1$ and $G_2$ and equip $G$ with a proper length function as in Remark \ref{remark:lengthamalga}. If $\alpha$ denotes the equivariant Hilbert space compression of $G$, then
\begin{enumerate}
\item $\alpha=1$ if $F$ is of index $2$ in both $G_1$ and $G_2$,
\item $\alpha=\alpha_1$ if $F=G_2$ and $\alpha=\alpha_2$ if $F=G_1$,
\item $ \alpha= \min(\alpha_1,\alpha_2, 1/2)$ otherwise.
\end{enumerate} 
\label{theorem:equivembed}}
\begin{proof}
Let us first focus our attention to the easy case where $F$ is of index $2$ in both $G_1$ and $G_2$. These assumptions imply that $F$ is a normal subgroup of both $G_1$ and $G_2$ and so it is a normal subgroup of $G$ with quotient $G_1/F * G_2/F=\Z_2*\Z_2$. Lemma \ref{cor:quotient} implies that $\alpha=1$.

Using the normal form theorem for amalgamated free products (see for example pg. 187 of \cite{LyndonSchupp}), it is easy to prove that the inclusion maps of the factors $G_1$ and $G_2$ into $G$ are quasi-isometric embeddings. This proves case $(2)$.

To prove $(3)$, one notices that $G$ cannot be amenable since it contains free subgroups.
Consequently, its equivariant compression is bounded from above by $1/2$ (See Theorem $5.3$ of \cite{guekam}). 
Using the fact that the inclusion maps of the factors $G_1$ and $G_2$ into $G$ are quasi-isometric embeddings, we conclude that $\alpha\leq \min(\alpha_1,\alpha_2, 1/2)$.\\

Conversely, assume $0\leq \epsilon < \min(\alpha_1,\alpha_2,1/2)$. There exist $C>0$ and conditionally negative definite functions
$\psi_i:G_i\rightarrow \R^+$  for $i\in \{1,2\}$, such that
\begin{enumerate}
\item $\psi_i$ is $F$-bi-invariant;
\item $\psi_i(f)=0, \ \forall f\in F$ and $\psi_i(x)\geq 1$ for all $x \in G_i\setminus F$
\item the $1$-cocycle $b_i$ associated to $\psi_i$ satisfies $(1/C) \ l_i(g)^{\epsilon} \leq \parallel b_i(g) \parallel \leq C\ l_i(g), \ \forall g\in G_i$.
\end{enumerate}
Let $R$ and $S$ be sets of representatives for the left cosets of $F$ in $G_1$ and $G_2$ respectively. Assume $1_{G_1}\in R, 1_{G_2}\in S$ and denote elements of $R, S$ and $F$ by $\alpha_i,\beta_j$ and $f$ respectively. According to \cite{Serre}, every element $x \in G$ has a unique normal form:
\begin{equation}
x=\alpha_1 \beta_1 \alpha_2 \beta_2 \ldots \alpha_k \beta_k f,
\label{eq:schrijfwijzeamalga}
\end{equation}
such that none of the $\alpha_i$ and $\beta_j$, except maybe for $\alpha_1$ or $\beta_k$, are equal to $1$.
In \cite{Haagerup} (see the proof of Proposition $6.2.3$),
 it is shown that the map $\psi:G\rightarrow \R^+$, defined by
\[ \psi(x)= \sum_{i=1}^k \psi_1(\alpha_i) + \sum_{j=1}^k \psi_2(\beta_j), \]
is a conditionally negative definite function on $G$. Application of Theorem \ref{theorem:cocconnegdef} gives an affine isometric action of $G$ on a Hilbert space $\h$ with $1-$cocycle $b$ satisfying $\parallel b(x) \parallel^2=\psi(x)$. Choose $x\in G$, and write $x=\alpha_1 \beta_1 \alpha_2 \beta_2 \ldots \alpha_l \beta_l f$ in normal form as above. We obtain that
\begin{eqnarray*}
 \parallel b(x) \parallel^2 &=& \psi(x)\\
 &=& \sum_{i=1}^k \psi_1(\alpha_i) + \sum_{j=1}^k \psi_2(\beta_j)\\
 &=& \sum_{i=1}^k \parallel b_1(\alpha_i) \parallel ^2 + \sum_{j=1}^k \parallel b_2(\beta_j) \parallel ^2
\end{eqnarray*}
Consequently,
\begin{equation}
\parallel b(x) \parallel \leq  C [ \sum_{i=1}^k l_1(\alpha_i)+ \sum_{j=1}^k l_2(\beta_j) ].
\label{eq:lipschitzamalga}
\end{equation}
Denote $l^{SB}(x)$ the {\em shortest blocklength} of $x$, meaning that it is the minimum of the lengths of all words representing $x$ which are of the form $(\gamma_1,\delta_1,\gamma_2,\ldots ,\gamma_k, \delta_k)$ where $k$ is as in Equation (\ref{eq:schrijfwijzeamalga}) and where the $\gamma_i$ and $\delta_i$ belong to $G_1$ and $G_2$ respectively. Take such a word, say $(\gamma_1,\delta_1,\gamma_2,\ldots ,\gamma_k, \delta_k)$, representing $x$. It follows from Bass-Serre theory that 
$\gamma_i \in C \alpha_i C$ and $\delta_i\in C \beta_i C$ for all $i\in \{1,2, \ldots, k\}$. Using Equation (\ref{eq:lipschitzamalga}), we obtain
\[ \parallel b(x) \parallel \leq C [l^{SB}(x) + 2M k] \leq C (\frac{2M}{M'} + 1) l^{SB}(x), \]
where $M=\max\{l_i(c)\mid c\in C, i=1,2\}$ and $M'=\min(1, \min\{l_1(\gamma)\mid \gamma \in G_1\setminus \{1\} \}, \min\{l_2(\delta)\mid \delta\in G_2\setminus \{1\} \})$. 

If $B\geq 1$ is such that $\forall g \in G_1: l_1(g)\leq B l(g)$ and  $\forall g\in G_2: l_2(g)\leq B l(g)$, then
$l^{SB}\leq B l$. We conclude that
\[ \parallel b(x) \parallel \leq (\frac{2M}{M'}+1) BC \ l(x), \ \forall x\in G, \]
and so $b$ is Lipschitz.

Conversely, we have
\[ \parallel b(x) \parallel ^2 \geq \sum_{i=1}^l (1/C^2) l_1(\alpha_i)^{2\epsilon} + \sum_{j=1}^l  (1/C^2) l_2(\beta_j)^{2\epsilon}. \]
Since $\forall a,b\in \R^+: a^{2\epsilon}+b^{2\epsilon}\geq (a+b)^{2\epsilon}$, we obtain
\[ \parallel b(x) \parallel^2 \geq (1/C^2) ( \sum_{i=1}^l  l_1( \alpha_i ) + \sum_{j=1}^l l_2( \beta_j ) ) ^{2\epsilon}. \]
Setting $\overline{M}=\sup_{f\in F} \{ l(f) \}$, we obtain
\begin{equation}
\parallel b(x) \parallel \geq (1/C) (l(x)-\min(l(x),\overline{M}))^{\epsilon} \geq (1/C') l(x)^{\epsilon} - D',
\label{eq:kleiner}
\end{equation}
for some $C'>0, D'\geq 0$.
It follows that the equivariant compression of $G$ is greater or equal than $\min(\alpha_1,\alpha_2,1/2)$. 
\end{proof}

We continue by proving a similar result for HNN-extensions $\HNN(H,F,\theta)$ where $F$ is such that $<F,\theta(F)>$ is a finite subgroup of $H$. Notice that this condition is satisfied whenever $F$ is finite and normal in $H$. We will need the following lemma; a proof can be found in \cite{Haagerup}, page 92.
{\lm Let $G$ be a discrete group acting (on the left) on a set $X$; let $H$ be a group, and let $c:X\times G \rightarrow H$ be a map verifying the cocycle relation
\begin{equation}
 c(x,g_1g_2)=c(x,g_1)c(g_1^{-1}x,g_2)
 \label{eq:coclem}
 \end{equation}
for all $x\in X$ and $g_1,g_2 \in G$. Let $\psi$ be a conditionally negative definite function on $H$, vanishing on a subset $A$ of $H$. Assume that for every $g \in  G $, the set $\{ x \in X:c(x,g)\notin A \}$ is finite; then the function $\tilde{\psi}$ on $G$ may be defined by 
\[ \tilde{\psi}(g)=\sum_{x\in X} \psi(c(x,g)) , \]
and $\tilde{\psi}$ is conditionally negative definite on $G$. \label{lm:cocyclehnn}}

Using similar ideas as in the previous proof, we prove the following result.
{\theorem Let $H$ be a countable group equipped with a proper length function and denote its equivariant Hilbert space compression by $\alpha_1$. Assume that $F$ is a subgroup of $H$ and that $\theta:F\rightarrow H$ is a group monomorphism such that the group generated by $\theta(F) \cup F$ is finite. Denote $G:=\HNN(H,F,\theta)$ and equip it with a proper length function as in Remark \ref{remark:HNNmetric}. Then, the equivariant Hilbert space compression $\alpha$ of $G:=\HNN(H,F,\theta)$ satisfies
\begin{enumerate}
\item $\alpha=1$ whenever $F=H$,
\item $\alpha=\min(\alpha_1,1/2)$ otherwise.
\end{enumerate}
\label{theorem:hnnequiv}}
The first claim follows trivially from Lemma \ref{cor:quotient}, but we have added it for completeness.
\begin{proof}
The fact that $F$ is a proper subset of $H$, implies that $G$ is not amenable and this forces $\alpha \leq 1/2$. It follows from Britton's lemma (see pg. 181 of \cite{LyndonSchupp}) that the inclusion $i:H \hookrightarrow G$ is quasi-isometric whenever $F$ is finite. Consequently, the equivariant compression of $H$ is another upper bound for $\alpha$. We obtain $\alpha \leq \min(1/2,\alpha_1 ) $ and proceed by showing that $\min(1/2, \alpha_1)$ is also a lower bound.

Denote the subgroup generated by $F\cup \theta(F)$ by $A$. Let $R$ and $S$ be sets of representatives for the left cosets of $F$ in $H$ and $\theta(F)$ in $H$ respectively, such that $1\in R$ and $1\in S$. We denote elements of $R, S, R\cup S, F$ and $A$ by $\alpha_i, \beta_i, \gamma_i, f$ and $a$ respectively. We denote the length on $H$ by $l_H$. From Bass-Serre theory, we know that every element $g \in G=\HNN(H,F,\theta)$ can be uniquely written in a normal form
\begin{equation}
g= \gamma_1 t^{i_1} \gamma_2 t^{i_2} \cdot \ldots \cdot \gamma_{k}t^{i_k} \alpha_{k+1} f,
\label{eq:schrijfwijzehnn}
\end{equation}
where $i_j=1$ whenever $\gamma_j\in R, i_j=-1$ whenever $\gamma_j\in S$ and no two subwords of the form $\gamma_1 t^{i_1} \gamma_2 t^{i_2} \cdot \ldots \cdot \gamma_{l}t^{i_l}$ with $l\leq k$ belong to the same left coset of $H$ in $G$.

If $\alpha_1 \neq 0$, then choose $0\leq \epsilon < \min(\alpha_1,1/2)$ and take an $A$-bi-invariant conditionally negative definite map $\psi$ on $H$ such that the associated $1-$cocycle $b$ on $H$ satisfies
\[ \exists C\geq 1, \ \forall h \in H:\   (1/C) \ l_H(x)^{\epsilon} \leq \parallel b(x) \parallel \leq C\ l_H(x).\]
We show that
\[ \tilde{\psi}: g \mapsto \sum_{i=1}^k \psi(\gamma_i) + \psi(\alpha_{k+1}), \]
where $g$ is written as in (\ref{eq:schrijfwijzehnn}), is a conditionally negative definite function on $G$.
To prove this, we remark that $G/ H$, the collection of left cosets of $H$ in $G$, can be identified with the elements whose normal form as in (\ref{eq:schrijfwijzehnn}) is of the form
\[ \gamma_1 t^{i_1} \gamma_2 t^{i_2} \cdot \ldots \cdot \gamma_{k}t^{i_k}. \]
This provides a section $\sigma:G/H \rightarrow G$ for the canonical projection map $\pi:G\rightarrow G/H$. Define $c:G/H \times G \rightarrow H$ by setting
\[ c(x,g)=\sigma(x)^{-1} g \sigma(g^{-1}x), \]
where $g^{-1}x$ stands for $\pi(g^{-1}\sigma(x))$. 
It is easy to check that $c$ satisfies Equation (\ref{eq:coclem}). We will apply Lemma \ref{lm:cocyclehnn} on $c$ to obtain $\tilde{\psi}$.
Therefore, choose any elements $g \in G$ and $x\in G/H$. Assume first that $g$, when written as in (\ref{eq:schrijfwijzehnn}), does not start with the word $\sigma(x)$. We write $\sigma(x)=x_0x_1$ and $g=x_0 y_1 \alpha_1 f$ where $x_0$ is the subword common to $\sigma(x)$ and $g$, and where $x_1$ ends with some non-zero power of $t$. Then
\[ g^{-1} \sigma(x)= f^{-1} \alpha_1^{-1} y_1^{-1} x_1 = y_1'x_1'a, \]
where $x_1'$ ends with some non-zero power of $t$ and where $a\in A$. This implies that
\[ \sigma(g^{-1} x )=y_1'x_1' \]
and so that $c(x,g)=a^{-1} \in A$. This already shows that for any $g\in G$, the set $\{ x\in G/H \mid c(x,g)\notin A \}$ contains only a finite number of elements.

Assume next that $g$ begins with the word $\sigma(x)$. Let us write
\[ g = \sigma(x) \gamma_l t^{i_l} \gamma_{l+1} t^{i_{l+1}} \cdot \ldots \cdot \gamma_k t^{i_k} \alpha_{k+1} f \]
as in (\ref{eq:schrijfwijzehnn}). Then
\[ g^{-1} \sigma(x) = (f^{-1} \alpha_{k+1}^{-1} t^{-i_k} \gamma_k^{-1} \cdot \ldots \cdot t^{-i_{l+1}} \gamma_{l+1}^{-1} t^{-i_l} )\gamma_l^{-1} = (\gamma_{k+1}'t^{-i_k} \gamma_k' \cdot \ldots \cdot t^{-i_{l+1}} \gamma_{l+1}'t^{-i_l} \tilde{a'} ) \gamma_l^{-1} ,\]
where $a'\in A$ and so
\[ \sigma(g^{-1} x )= \gamma_{k+1}'t^{-i_k} \gamma_k' \cdot \ldots \cdot t^{-i_{l+1}} \gamma_{l+1}'t^{-i_l}. \]
Therefore $c(x,g)=\gamma_l \tilde{a'}^{-1}$ and $\psi(c(x,g))=\psi(\gamma_l \tilde{a'}^{-1})=\psi(\gamma_l)$.
By Lemma \ref{lm:cocyclehnn}, we conclude that $\tilde{\psi}$ is conditionally negative definite.

Next, consider the Bass-Serre tree $T$ associated to the HNN-extension $G$. The vertex set is $G/H$. Let $d_T(g_1H,g_2H)$ be the tree distance between vertices $g_1H,g_2H$. It is known that, on a tree, the distance is a conditionally negative definite kernel (\cite{propriété(T)}, Proposition $2$ in $\mathsection 6.a$), so that $\psi'(g)=d_T(H,gH)$ defines a conditionally negative definite function on $G$. Set $\overline{\psi}=\tilde{\psi}+\psi'$. 

Theorem \ref{theorem:cocconnegdef} associates to $\overline{\psi}$ a $1-$cocycle $\overline{b}$ relative to some affine isometric action of $G$ on a Hilbert space $\h$. Let us start by proving that $\overline{b}$ is Lipschitz. Choose any $g\in G$ and write
\[ g= \gamma_1 t^{i_1} \gamma_2 t^{i_2} \cdot \ldots \cdot \gamma_{k}t^{i_k} \alpha_{k+1} f \]
 as in equation (\ref{eq:schrijfwijzehnn}). We have that
\begin{eqnarray*}
\parallel \overline{b}(g) \parallel^2 &=& \overline{\psi}(g) \\
&=& (\sum_{i=1}^k \psi(\gamma_i)) + \psi(\alpha_{k+1}) + \psi'(g) \\
&=& (\sum_{i=1}^k \parallel b(\gamma_i) \parallel ^2) + \parallel b(\alpha_{k+1}) \parallel ^2 + \psi'(g) \\
& \leq & C^2 [ (\sum_{i=1}^k l_H(\gamma_i)^2) + l_H(\alpha_{k+1})^2 + d_T(H,gH) ].
\end{eqnarray*}
where $l_H$ is the length function on $H$ and $C\geq 1$. Consequently,
\begin{equation}
\parallel \overline{b}(g) \parallel \leq C [\sum_{i=1}^k l_H(\gamma_i) + l_H(\alpha_{k+1}) + d_T(H,gH)].
\label{eq:shortestblocks}
\end{equation}
Now, write $g=h_1 t^{i_1} h_2 t^{i_2} \cdot \ldots \cdot h_k t^{i_k} h_{k+1}$ with $h_j\in H, \ \forall j \in \{1,2,\ldots ,k \}$ and with $k$ and the $i_j$ as in Equation (\ref{eq:schrijfwijzehnn}). It follows from Bass-Serre theory that $h_j \in A \gamma_j A , \ \forall j \in \{1,2, \ldots, k\}$ and $h_{k+1} \in A \alpha_{k+1} f A$. 

Denote $l^{SB}(g)$ the {\em shortest blocklength} of $g$, meaning that it is the length of $g$ looking only at the representatives of $g$ in $H*\Z$ of the form $h_1 t^{i_1} h_2 t^{i_2} \cdot \ldots \cdot h_k t^{i_k} h_{k+1}$, where $k$ and the $i_j$ are as in Equation (\ref{eq:schrijfwijzehnn}). Take $M=\max\{l_H(a)\mid a \in A \}$ and $M'=\min(1,\min\{l_H(h) \mid h\in H\setminus \{1\} \})$. Together with Equation (\ref{eq:shortestblocks}), we get that
\[ \parallel \overline{b}(g) \parallel \leq C[l^{SB}(g) + 2M (k+1)] \leq C(\frac{2M}{M'}+1) l^{SB}(g).\]
If $B\geq 1$ is such that $\forall h\in H: l_H(h)\leq B l(h)$, then
$l^{SB}\leq B l$. We conclude that
\[ \parallel \overline{b}(g) \parallel \leq C (\frac{2M}{M'}+1) B \ l(g), \ \forall g\in G , \]
so that $\overline{b}$ is Lipschitz.
Conversely, we have that 
\[ \parallel \overline{b}(g) \parallel^2 \geq (1/C^2) [ \sum_{i=1}^k (l_H(\gamma_i)^{2\epsilon}) + l_H(\alpha_{k+1})^{2\epsilon} + d_T(H,gH) ]. \]

Since $2\epsilon \leq 1$, we get
\begin{eqnarray*}
\parallel \overline{b}(g) \parallel & \geq & (1/C) ((\sum_{i=1}^k l_H(\gamma_i)) + l_H(\alpha_{k+1}) + d_T(H,gH))^{\epsilon} \\
& \geq & (1/C) (l(g) - \min(M,l(g)))^{\epsilon}. \\
\end{eqnarray*}

Since $\epsilon$ was any number between $0$ and $\min(\alpha_1,1/2)$, we conclude that the equivariant compression of $G$ exceeds $\min(1/2,\alpha_1)$ and so we have proven that $\alpha = \min(1/2,\alpha_1)$.

\end{proof} 
\noindent
{\bf Acknowledgements}\\

\noindent
I thank Alain Valette for reading previous versions of the article, for interesting conversations and for pointing out some open problems related to Hilbert space compression. I thank Alain Valette and Paul Igodt for encouragements during this work. I thank Nansen Petrosyan for interesting conversations.


\begin{thebibliography}{99}
\bibitem{Gromov}
M. Gromov, Asymptotic invariants of infinite groups, Geometric Group Theory (A. Niblo and M. Roller, eds.), London Mathematical Society Lecture Notes 182 (1993), 1–-295.

\bibitem{Gromov2}
M. Gromov, Problems (4) and (5), Novikov conjectures, index theorems and rigidity,
Lecture Note Series, London Mathematica Society 226, Vol.1 (1995) (S. Ferry, A. Ranicki and J. Rosenberg, editors)

\bibitem{Yu}
G. Yu, The coarse Baum-Connes conjecture for spaces which admit a uniform embedding into Hilbert space, 
Invent. Math. 139, no. 1 (2001), 201--240.

\bibitem{Yu:2}
G. Skandalis, J. L. Tu, and G. Yu,
Coarse Baum-Connes conjecture and groupoids, Topology 41 (2002), 807--834.

\bibitem{guekam}
E.Guentner, J.Kaminker,
Exactness and uniform embeddability of discrete groups, Journal of the London Mathematical Society 70, no.3 (2004), 703--718.

\bibitem{Bourgain}
J. Bourgain,
On Lipschitz embedding of finite metric spaces in Hilbert space,
Israel J. Math. 52 (1985), 46–-52.

\bibitem{Arz}
G. Arzhantseva, C. Drutu, M. Sapir,
Compression functions of uniform embeddings of groups into Hilbert and Banach spaces,
Journal f\"ur die reine und angewandte Mathematik 633 (2009), 213--235. 

\bibitem{naoper}
A. Naor, Y.Peres, Embeddings of discrete groups and the speed of random walks, International Mathematics Research Notices 2008, Art. ID rnn 076. 

\bibitem{inbedmap}
X. Chen, M. Dadarlat, E Guentner, G. Yu,
Uniform embeddability and exactness of free products,
 Journal of functional analysis 205, No.1 (2003), 168--179.

\bibitem{hyper}
O. Kharlampovich and A. Myasnikov, hyperbolic groups and free constructions, Transactions of the American mathematical society 350, No.2 (1998), 571--613.

\bibitem{tess}
R. Tessera, Asymptotic isoperimetry on groups and uniform embeddings into Banach spaces,
to appear in Commentarii Mathematici Helvetici.

\bibitem{quasisom}
P. Papasoglu, K. Whyte, Quasi-isometries between groups with infinitely many ends, Commentarii Mathematici Helvetici 77, No.1 (2002), 133--144.

\bibitem{guedad}
M. Dadarlat and E. Guentner,
Constructions preserving Hilbert space uniform embeddability of discrete groups,
Transactions of the American mathematical society 355, No.8 (2003) 3253--3275.

\bibitem{propriété(T)}
P. de la Harpe, A. Valette,
La propri\'et\'e (T) de Kazhdan pour les groupes localement compacts (avec un appendice de Marc Burger), Ast\'erisque 175, Soci\'et\'e Math\'ematique de France, Paris, 1998.

\bibitem{Haagerup}
P. A. Cherix, M. Cowling, P. Jolissaint, P. Julg, A. Valette,
Groups with the Haagerup Property,
Progress in Mathematics 197, 2001.

\bibitem{LyndonSchupp}
R.C. Lyndon, P.E. Schupp, Combinatorial group theory, Ergebnisse der Mathematik und ihrer Grenzgebiete 89, A Series of Modern Surveys in Mathematics.

\bibitem{Serre}
J.P. Serre,
`Arbres, amalgames, $SL_2$', Ast\'erisque 2, Soci\'et\'e Math\'ematique de France, Paris, 1977.

\end{thebibliography}
\end{document}